\documentclass[12pt]{amsart}
\usepackage{amssymb}
\usepackage{mathrsfs}
\usepackage{amsthm}
\usepackage{fullpage}
\usepackage{graphicx}
\newtheorem{lemma}{Lemma}[section]
\newtheorem{theorem}{Theorem}[section]
\newtheorem*{theorem*}{Theorem}
\newtheorem{corollary}{Corollary}[section]
\theoremstyle{definition}
\newtheorem{example}{Example}[section]
\numberwithin{equation}{section}
\title{Free probability of type B and asymptotics of finite-rank perturbations of random matrices}
\author{D. Shlyakhtenko}
\address{Department of Mathematics, UCLA, Los Angeles, CA 90095}
\email{shlyakht@math.ucla.edu}
\thanks{Research supported by the National Science Foundation grant DMS-1500035.}

\begin{document}
\begin{abstract}
We show that finite rank perturbations of certain random matrices fit in the framework of infinitesimal (type B) asymptotic freeness.  This can be used to explain the appearance of free harmonic analysis (such as subordination functions appearing in additive free convolution) in computations of outlier eigenvalues in spectra of such matrices.
\end{abstract}
\maketitle
\section{Introduction}
Voiculescu's free probability theory \cite{dvv:book,DVV:random} has been remarkably successful in providing a framework for understanding asymptotic behavior of families of large random matrices.  To give a concrete example, let $A(N)$ be a self-adjoint $N\times N$ matrix chosen at random (with respect to the Haar measure of this symmetric space) among all matrices with prescribed eigenvalues $\lambda^A_1(N)\leq \dots\leq \lambda^A_N(N)$ and let $B(N)$ be another independently chosen matrix with eigenvalues $\lambda_1^B(N)\leq \dots\leq \lambda^B_N(N)$.  Let $$\eta^A_N = \frac{1}{N} \sum_{j=1}^N \delta_{\lambda^A_j(N)},\qquad \eta^B_N =\frac{1}{N} \sum_{j=1}^N \delta_{\lambda^B_j(N)}$$ be their empirical spectral measures.  Assuming that $\eta^A_N \to \eta_A$ and $\eta^B_N \to \eta_B$ converge weakly to some compactly supported measures, the matrices $A(N)$ and $B(N)$ become freely independent in the limit $N\to \infty$.  This means that free probability can be used to deduce limits of empirical spectral measures of various expressions in $A(N)$ and $B(N)$, such as $A(N)B(N)A(N)$, $A(N)B(N)+B(N)A(N)$, etc.  For example, if $C(N)  = A(N)+B(N)$ and we let $$\eta^C_N = \frac{1}{N} \mathbb{E}
\left[ \sum_{\lambda \textrm{ eigenvalue of } C} \delta_\lambda \right]$$ be its empirical spectral measure, then $\eta^C_N \to \eta^C$ and $\eta^C$ is given by the free additive convolution $$\eta^C = \eta^A \boxplus \eta^B.$$  Such free convolution can be effectively computed using complex analysis tools introduced by Voiculescu.  In particular, if we denote by 
$$G_\eta(z) = \int \frac{1}{z-t} d\eta(t)$$ the Cauchy transform of the measure $\eta$, the following analytic subordination result holds:
\begin{theorem*}\cite{biane:subordination,dvv:entropy1}
Let $\eta_1, \eta_2$ be probability measures on $\mathbb R$ and let $\eta = \eta_1 \boxplus \eta_2$. Denote by $\mathbb{C}^+$ the upper half-plane $\{z\in \mathbb{C}: \operatorname{Im}(z)>0\}$.  Then there exist analytic maps $\omega_1, \omega_2 : \mathbb{C}^+\to \mathbb{C}^+$ which are uniquely determined by the following properties:\begin{itemize}
\item[(1)] $G_{\eta_1} (\omega_1(z)) = G_{\eta_2}(\omega_2(z)) = G_\eta (z)$, $z\in \mathbb{C}^+$
\item[(2)] $\omega_1(z) + \omega_2(z) = z + 1/G_\eta (z)$ 
\item[(3)] $\lim_{y\to+\infty} \omega_j (iy) / (iy) =\lim_{y\to\infty} \omega'_j(iy) =1$, $j=1,2$. \end{itemize}
\end{theorem*}

More recently, the question of finite-rank perturbation of certain random matrices has attracted considerable attention (an incomplete bibliography includes \cite{spikes,spikes0,spikes1,spikes2,spikes3,spikes4,spikes5,spikes6,spikes7,spikes8,spikes9,spikes10,spikes11,spikes12}).  To give a concrete example, let $A(N)$ be chosen at random to have eigenvalues $\Lambda$ and let $B(N)$ be a rank $1$ self-adjoint $N\times N$ matrix with the sole nonzero eigenvalue $\lambda$.  Assuming once again that $\eta^A_N \to \eta$, the question is to understand the asymptotic behavior of $C(N)=A(N)+B(N)$. Because $B(N)$ is rank $1$, its limiting distribution in the sense of free probability theory is that of the zero matrix and so there is no difference in the limit between the law of $A(N)+B(N)$ and $A(N)+0$.  Nonetheless, not only does a detailed description of the eigenvalue distribution of $C(N)$ exist, but its description surprisingly involves the subordination functions that appeared in the previous theorem \cite{spikes4,spikes}.  This strongly suggests that free independence is still involved in the description of the empirical spectral measure of $C(N)$, although perhaps in a non-trivial way.  Our main result gives an explanation of this phenomenon.

The main idea of the present paper is to examine not just the empirical law of $C(N)$ but also its $1/N$ correction (this has been already explored previously by \cite{Johanssen, Edelmann, belinschi-shlyakht:typeB}).  In other words, we consider $\eta^C_N$, the empirical spectral measure of $C(N)$ and write $$\eta^C_N = \eta_C+ \frac{1}{N} \eta_C' + o(N^{-1})$$ for a measure $\eta_C$ and a distribution $\eta_C'$.  The same can be done for matrices $A(N)$ and $B(N)$.  In our concrete example, $$\eta^A_N = \eta^A + 0 \frac{1}{N} + O(N^{-2})$$ and $$\eta^B_N = \delta_{0} + \frac{1}{N} (\delta_\lambda - \delta_0).$$ We then show that $A(N)$ and $B(N)$ are asymptotically free in the following sense.  Suppose that we replace $A(N)$ and $B(N)$ by some other variables $\hat{A}(N)$ and $\hat{B}(N)$ having the same laws as $A(N)$ and $B(N)$, respectively. Suppose that we further assume that $\hat{A}(N)$ and  $\hat{B}(N)$ are freely independent (for each $N$).  Letting $\hat{C}(N)= \hat{A}(N)+\hat{B}(N)$, our main result implies that the empirical spectral measures of $C(N)$ and $\hat{C}(N)$ are the same up to order higher than $1/N$.  In other words, $$\eta^{C}_N = \eta^A_N \boxplus \eta^B_N + o(1/N)$$ which in our concrete case becomes
$$\eta^C_N = \eta^A \boxplus \left( \delta_{0} + \frac{1}{N} (\delta_\lambda - \delta_0)\right) + o(1/N).$$  Since $\hat{A}(N)$ and $\hat{B}(N)$ are free, the computation of the law of $\hat{C}(N)$ involves free convolution, subordination functions, etc.  Thus we obtain an explanation as to why these objects are also involved in the description of the law (i.e., empirical spectral measure) of $C(N)$.

To phrase this result in terms most resembling Voiculescu's asymptotic freeness for random matrices, it is useful to consider the framework of {\em infinitesimal non-commutative probability spaces} introduced in \cite{belinschi-shlyakht:typeB} in connection with so-called type B free probability theory of Biane, Goodman and Nica \cite{BGN} (see also \cite{nica-fevrier:typeB}).  Such a space consists of a unital $*$-algebra $A$ and a pair of linear functionals $\phi,\phi':A\to \mathbb{C}$ so that $\phi(1)=1$ and $\phi'(1)=0$.  The idea behind the definition is that we might be given a family of unital linear functionals $\phi_t: A\to \mathbb{C}$ so that $\phi = \phi_{t=0}$ and $\phi' = \frac{d}{dt}\big\vert_{t=0} \phi_t$; in other words, an infinitesimal probability space captures the zeroth and first order behavior in $t$ of a family of laws $\phi_t$ on $A$.  One can then formulate a notion of freeness appropriate for this setting.  In essence, $A_1, A_2\subset A$ are infinitesimally free if Voiculescu's freeness condition holds to zeroth and first order in $t$. 
This means that one can manipulate variables form $A_1$ and $A_2$ {\em as if they were freely independent}, up to an error of higher order than $t$.  For example, the law of the sum $c= a+b$ of $a\in A_1$ and $b\in A_2$ can be described by keeping terms of order $1$ and $t$ in the free convolution of the laws of $a$ and $b$.  This was worked out in \cite{belinschi-shlyakht:typeB} where we showed that such formulas naturally involve subordination functions and their derivatives.

Returning to the matrices $A(N)$ and $B(N)$ we note that they can be viewed as infinitesimal random variables if one sets $t=N^{-1}$.   Our main result then states that these matrices are asymptotically free in the infinitesimal sense.  In particular, it follows that the appearance of subordination functions in the description of outliers is no accident: they precisely arise because of freeness (``to order $1/N$'') of $A(N)$ and $B(N)$.  Our results imply that free probability can be used to describe asymptotics of empirical spectral measures of other more complicated expressions involving random matrices and finite-rank matrices.  While the description is less precise than what is available for random matrices (where more refined eigenvalue statistics are known and one can control extremely well the positions of outliers) it is worth pointing out that our ``softer'' approach works for arbitrary polynomials in matrices and arbitrary finite rank perturbations.

The remainder of the paper is organized as follows.  We first review the necessary background on infinitesimal freeness and free probability of type B. We then show that certain types of random matrices are infinitesimally free asymptotically.  Finally, we consider examples explaining how computations with free convolution (of type B) match up with computations of spectral measure and outliers of finite-rank perturbations of unitarily invariant matrices and self-adjoint complex and real Gaussian random matrices.

\section{Infinitesimal freeness.}

An infinitesimal non-commutative probability space \cite{belinschi-shlyakht:typeB,nica-fevrier:typeB} consists of a unital $*$-algebra $A$ and a pair of linear functionals $\phi,\phi':A\to \mathbb{C}$ so that $\phi(1)=1$ and $\phi'(1)=0$.  The idea behind the definition is that we might be given a family of unital linear functionals $\phi_t: A\to \mathbb{C}$ so that $\phi = \phi_{t=0}$ and $\phi' = \frac{d}{dt}\big\vert_{t=0} \phi_t$; in other words, an infinitesimal probability space captures the zeroth and first order behavior in $t$ of a family of laws $\phi_t$ on $A$.  This notion was introduced to give an alternative interpretation of free probability of type B introduced by Biane, Goodman and Nica \cite{BGN} for purely combinatorial reasons (they wanted to obtain an analog of free probability theory in which the role of non-crossing partitions is played by the lattice of type B non-crossing partitions). 

As already considered in \cite{belinschi-shlyakht:typeB}, there is a natural structure of an infinitesimal non-commutative probability space for algebras of random matrices.  Indeed, let us assume that $A_1(N),\dots,A_k(N)$ are random matrices of size $N\times N$, $N=N_0, N_0+1,\dots$, let $A$ be the algebra of polynomials in $k$ non-commuting indeterminates and define $\tau_{N} : A\to \mathbb{C}$ by $$\tau_{N} (P) = \mathbb{E} \frac{1}{N} Tr(P(A_1(N),\dots,A_k(N))),\qquad P\in A.$$  We can then set 
\begin{eqnarray*}
\tau &=& \lim_{N\to\infty} \tau_{N} \\
\tau' &=& \lim_{N\to\infty} N (\tau_{N} - \tau)
\end{eqnarray*}
provided that these limits exist.  This has been extensively studied previously in the context of single matrix models by Johanssen \cite{Johanssen} and later by Dumitriu and Edelmann \cite{Edelmann}.

\begin{example} (i) Let $A_j(N)$ be complex Gaussian random matrices.  In this case both limits exist: $\tau$ is the free semicircle law \cite{DVV:random} and $\tau'=0$ since an easy computation shows that $\tau_N-\tau = O(1/N^2)$ (this is no longer the case for real or symplectic Gaussian matrices, cf.  \cite{Johanssen,Edelmann}).  \\
(ii) Let $F$ be a fixed $N_0\times N_0$ self-adjoint matrix with eigenvalues $\theta_1,\dots,\theta_{N_0}$ regarded as an $N\times N$ matrix by padding with zeros.  Then $\tau(P)=P(0)$ and $\tau'(P) = \sum_{j=0}^{N_0}( P(\lambda_j) - P(0))$. 
\end{example}

\subsection{Definition of infinitesimal freness}
One says that two unital subalgebras, $A_1, A_2 \subset A$ are {\em infinitesimally free} with respect to $\phi,\phi'$ if the following two conditions hold whenever $a_1,\dots,a_r\in A$ are such that $a_k \in A_{i_k}$, $i_1\neq i_2$, $i_2\neq i_3$, $\dots$ and $\phi(a_1)=\phi(a_2)=\cdots=\phi(a_n)=0$:
\begin{eqnarray*}
\phi (a_1\cdots a_r)  &=& 0; \\
\phi'(a_1\cdots a_r)  & = & \sum_{j=1^r} \phi(a_1\cdots a_{j-1} \phi'(a_j) a_{j+1}\cdots a_r).
\end{eqnarray*}
These two conditions are equivalent to the requirement that if we set $\phi_t = \phi + t \phi'$, then for any $a_k \in A_{i_k}$ so that $i_1\neq i_2$, $i_2\neq i_3$, $\dots$, we have that $\phi_t( (a_1 - \phi_t(a_1)) \cdots (a_r - \phi_t(a_r))) = o(t)$, i.e., $A_1$ and $A_2$ are ``free to zeroth and first order''. Note that the first condition simply states that $A_1$ and $A_2$ are (in the usual sense) free with respect to $\phi$. 

\subsection{A criterion for infinitesimal freeness}

The following lemma will be useful in checking asymptotic infinitesimal freeness.

\begin{lemma}\label{lemma:infFreeConditions}
Assume that $a_1,\dots,a_n, e_1,\dots,e_m$ are elements of an infinitesimal probability space $(A,\phi,\phi')$ with the following properties: \begin{itemize}
\item[(i)] For any $p$ in the non-unital algebra generated by $e_1,\dots,e_m$, $\phi(p)=0$.  
\item[(ii)] $\phi(c_1\dots c_r) = \phi(c_r c_1 \dots c_{r-1})$ and $\phi'(c_1 \dots c_r) = \phi'(c_r c_1 \dots c_{r-1})$ for any $r$ and any $c_1,\dots,c_r \in A$. \end{itemize}

Then $(a_1,\dots,a_n)$ and $(e_1,\dots,e_m)$ are infinitesimally free iff for any $E_j $ in the non-unital algebra generated by $e_1,\dots,e_m$ and $Q_j \in \operatorname{Alg}(a_1,\dots,a_n)$ so that $\phi(Q_j)=0$ and any $r\geq 1$, \begin{itemize}
\item[(a)] $\phi (E_1 Q_1 E_2 Q_2 \dots E_r Q_r) = 0$ and 
\item[(b)] $
\phi'(E_1 Q_1 E_2 Q_2 \dots E_r Q_r) = 0$.
\end{itemize}
\end{lemma}

\begin{proof}
We first note that $(a_1,\dots,a_n)$ and $(e_1,\dots,e_m)$ are free with respect to $\phi$.  Indeed, this follows right away from condition (a) and assumption (iii). 

Next, we note that assumption (b) recursively determines $\phi'$ on all of the algebra generated by $e_1,\dots,e_m$ and $a_1,\dots,a_n$.  Thus it is sufficient to check that (b) holds under the assumption that infinitesimal freeness holds.

In that case, 
\begin{equation}\label{eq:phiPrime}
\begin{aligned}
\phi'(E_1 Q_1 E_2 Q_2 \dots E_r Q_r) &= \sum_j \phi(E_1 Q_1 \dots E_{j-1} Q_{j-1} \phi'(E_j) Q_{j} E_{j+1}Q_{j+1} \dots E_{n}Q_n )
\\ & \qquad + \sum_{j} \phi(E_1 Q_1 \dots E_{j} \phi'(Q_j) E_{j+1} \dots E_n Q_n).
\end{aligned}
\end{equation}  
The terms $\phi(E_1 Q_1 \dots E_{j} \phi'(Q_j) E_{j+1} \dots E_n Q_n) = \phi'(Q_j) \phi(E_1 Q_1 \dots E_{j} E_{j+1} \dots E_n Q_n)$ are zero by freeness, since by assumption $\phi(E_j E_{j+1})=0$.  Thus the second sum in \eqref{eq:phiPrime} is  zero.  

Returning to the first term, 
if $r> 1$, we have that 
\begin{multline*}
\phi(E_1 Q_1 \dots E_{j-1} Q_{j-1} \phi'(E_j) Q_{j} E_{j+1}Q_{j+1} \dots E_{n}Q_n ) = \\
\phi'(E_j) \phi(E_1 Q_1 \dots E_{j-1} (Q_{j-1}  Q_{j}-\phi(Q_{j-1}Q_j)) E_{j+1}Q_{j+1} \dots E_{n}Q_n )
 \\ + \phi (Q_{j-1}  Q_{j}) \phi'(E_j) \phi(E_1 Q_1 \dots E_{j-1}  E_{j+1}Q_{j+1} \dots E_{n}Q_n )\end{multline*}
which is zero by assumption (a), since at least one term in the monomial comes from the non-unital subalgebra generated by $e_1,\dots,e_m$.  If $r=1$, we have 
$$ \phi'(E_1 Q_1) = \phi'(E_1) \phi(Q_1)=0$$ by our assumption that $\phi(Q_j)=0$.
\end{proof}

\section{Asymptotic infinitesimal freeness for random matrices}

We begin by describing two ensembles considered in the present paper, the unitary invariant ensemble having deterministic eigenvalues and complex and real Gaussian ensembles.  We prove a technical lemma for each of them.  We then state and prove our main result in \S\ref{sec:AsymptFree}.

Let us denote by $E_{ij}$ the matrix all of whose entries are zero, except that the $i,j$-th entry is $1$.

\subsection{Unitarily invariant matrices.}  \label{sec:notation} Let $A_1(N),\dots,A_k(N)$ be self-adjoint $N\times N$ random matrices of the form $A_j(N)=U \Lambda_j(N) U^*$ where $\Lambda_1(N),\dots,\Lambda(N)_k$ are deterministic matrices and  $U$ is an $N\times N$ unitary matrix randomly chosen with respect to the Haar measure.  We require  that  for any polynomial $q$, $N^{-1} \mathbb{E} Tr(q(A_1(N),\dots,A_k(N)))$  converges as $N\to\infty$ to a limit denoted by $\tau(q)$.  We  moreover assume that $\sup_{j,N} \Vert A_j(N)\Vert_\infty <\infty$. 

\begin{lemma}\label{lemma:limit} With the notation and assumptions of \S\ref{sec:notation},
for any polynomials $q_1,\dots,q_r$ in $k$ indeterminates,
\begin{gather*} 
\lim_{N\to\infty} \mathbb{E} Tr ( E_{i_r j_1} q_1(A_1(N),\dots,A_k(N))  E_{i_1 j_2} q_2(A_1(N),\dots,A_k(N))E_{i_2 j_3} \times \\ \cdots \times E_{i_{r-1} j_r} q_r(A_1(N),\dots,A_k(N)))
= \prod_{s=1}^r \delta_{j_s = i_s}\tau(q_s) 
\end{gather*} 
\end{lemma}

\begin{proof} We note that 
\begin{gather*}
\mathbb{E}Tr ( E_{i_r j_1} q_1(A_1(N),\dots,A_k(N)) E_{i_1 j_2} q_2(A_1(N),\dots,A_k(N))E_{i_2 j_3}\times \\ \qquad \qquad \qquad \cdots\times  E_{i_{r-1} j_r} q_r(A_1(N),\dots,A_k(N))) 
= \mathbb{E} \prod_{s=1}^r [ q_s(A_1(N),\dots,A_k(N))]_{j_s, i_s} .
\end{gather*}
Because of unitary invariance, $$\mathbb{E}[q_s(A_1(N),\dots,A_k(N))]_{a,b} = \mathbb{E}\frac{1}{N} \sum_{u=1}^N [q_s(A_1(N),\dots,A_k(N))]_{a+u,b+u},$$ where the addition of indices is modulo $N$.   If we denote by $\Sigma(a,b)$ the matrix $\sum_{u} E_{a+u,b+u}$ (addition mod $N$), we further have 
$$\mathbb{E} q_s(A_1(N),\dots,A_k(N))]_{a,b} = \mathbb{E} \frac{1}{N} Tr(q_s(A_1(N),\dots,A_k(N))\Sigma(a,b)).$$
By \cite{DVV:random} unitary invariance implies that $$\lim_{N\to\infty} \mathbb{E} \frac{1}{N} Tr(q_s(A_1(N),\dots,A_k(N))\Sigma(a,b))=\tau(q) \lim_{N\to\infty} Tr(\Sigma(a,b)) = \tau(q_s)\delta_{a=b}.$$  

Let $\Vert\cdot\Vert_{HS}$ denote the Hilbert-Schmidt norm on $N\times N$ matrices given by $$\Vert B\Vert_{HS} = (Tr(B^*B))^{1/2} = (\sum_{ij=1}^N |B_{ij}|^2)^{1/2}.$$  Note that  $$[q_s(U\Lambda_1(N)U^*,\dots,U\Lambda_k(N)U^*)]_{a,b} 
 = [Uq_s(\Lambda_1(N),\dots,U\Lambda_k(N))U^*]_{a,b} = [ U Q U^*]_{a,b}$$ where $Q=q_s(\Lambda_1(N),\dots,\Lambda_k(N))$ is bounded in the operator norm.  It follows that if $U'$ is another unitary, then
 \begin{eqnarray*}
 \big|[ U Q U^*]_{a,b} - [ U' Q (U')^*]_{a,b} \big | &\leq &
\Vert  U Q U^* - U' Q (U')^* \Vert_{HS} 
\\ &\leq & \Vert (U-U') QU^*\Vert_{HS} + \Vert U' Q(U-U')^* \Vert_{HS}\\ 
&\leq & 2\Vert U-U'\Vert_{HS} \Vert Q\Vert_{\infty}.\end{eqnarray*}
Thus the function $U\mapsto  [q_s(U\Lambda_1(N)U^*,\dots,U\Lambda_k(N)U^*)]_{a,b} $ is Lipschitz (of bounded constant) as a function on the space $N\times N$ unitaries with the Hilbert-Schmidt norm coming from the non-normalized trace on matrices.  Therefore, because of concentration (see e.g. \cite[Chapter 6]{guionnet:randomBook})
$$
\mathbb{E} \prod_{s=1}^r [ q_s(A_1(N),\dots,A_k(N))]_{j_s, i_s} = \prod_{s=1}^r \mathbb{E} [ q_s(A_1(N),\dots,A_k(N))]_{j_s, i_s},$$ which concludes the proof.
\end{proof}

\subsection{Gaussian random matrices}  \label{sec:GaussianNotations} 
Let $A_1(N),\dots,A_k(N)$ be self-adjoint $N\times N$ random matrices with entries $( [A_r(N)]_{i,j} : 1\leq r\leq k, 1\leq i\leq N)$  independent random variables each of variance $ (1+\delta_{ij}) N^{-1/2}$, and so that $[A_r(N)]_{i,j}$ are real or complex Gaussian if $i\neq j$ and real Gaussian if $i=j$.    The following is an analog of Lemma~\ref{lemma:limit}; one could prove it using concentration and the fact that maximal and minimal eigenvalues of real or complex Gaussian matrices are bounded with high probability.  However, we prefer to give a direct combinatorial proof.  A version of this proof also works for squares of rectangular matrices.

\begin{lemma}\label{lemma:limitGaussian} Let $A_1(N),\dots,A_k(N)$ be as in  \S\ref{sec:GaussianNotations},
for any polynomials $q_1,\dots,q_r$ in $k$ indeterminates,
\begin{gather*} 
\lim_{N\to\infty} \mathbb{E} Tr ( E_{i_r j_1} q_1(A_1(N),\dots,A_k(N))  E_{i_1 j_2} q_2(A_1(N),\dots,A_k(N))E_{i_2 j_3} \times \\ \cdots \times E_{i_{r-1} j_r} q_r(A_1(N),\dots,A_k(N)))
= \prod_{s=1}^r \delta_{j_s = i_s}\tau(q_s) 
\end{gather*} 
\end{lemma}
\begin{proof}
Let $q_1,\dots,q_r$ be monomials, so that 
\begin{eqnarray*}
[q_p (A_1(N),\dots,A_k(N))]_{a,b} &=& \sum_{J=(j_1,\dots,j_{d_p-1})} [A_{u_1}(N)]_{a,j_1} [A_{u_2}(N)]_{j_1,j_2} \cdots [A_{u_{d_p}}(N)]_{j_{d_p-1},b} \\
&=& \sum_{J} A_p(a,J,b).
\end{eqnarray*}

 It follows that with this notation, 
 \begin{align*}
Tr (&E_{i_r j_1}  q_1(A_1(N),\dots,A_k(N))  E_{i_1 j_2} q_2(A_1(N),\dots,A_k(N))E_{i_2 j_3} \times  \\ & \qquad \cdots \times  E_{i_{r-1} j_r} q_r(A_1(N),\dots,A_k(N))) 
\\  & = 
 \sum_{J_1,\dots,J_r} \mathbb{E}\big(A_1(j_1,J_1,i_1) A_2(j_2,J_2,i_2) \cdots A_r(j_r, J_r, i_r)\big).
 \end{align*}

By a standard argument involving the  Wick formula (see e.g. \cite[Chapter 1]{guionnet:randomBook}), taking into account the value of the variance of $[A_k(N)]_{ij}$, we see that each term $$\mathbb{E}\big(A_1(j_1,J_1,i_1) A_2(j_2,J_2,i_2) \cdots A_r(j_r, J_r, i_r)\big)$$ has either value $0$ or $N^{d/2}$ where $d$ is the sum of the degrees of the monomials $q_1,\dots,q_r$.  Moreover, nonzero terms are in one-to-one correspondence to labeled triangulations $\mathscr{T}$ of an orientable two-dimensional surface having one zero-cell and $d$ one-cells (edges), and a certain number of faces.  The labeling is as follows. Number the half-edges emanating from the zero-cell from $1$ to $d$.  Then label (clockwise) the top and bottom of each half-edge by indices from the ordered set $\{j_1\} \cup J_1 \cup \{i_1\} \cup \{j_2\} \cup J_2 \cup \cdots \cup J_r \cup \{i_r\}$.  The associated term is zero unless indices agree along every edge (although even some terms satisfying this requirement may be non-zero; there are further requirements involving matching various terms from the monomials $q_1,\dots,q_r$ which we will for now ignore). 

 We see from this structure that for any face $F$ the following must hold.  Either the boundary edges of $F$ are not labeled by any of $j_1,\dots,j_r,i_1,\dots,i_r$, in which case the indices along the boundary of that face must all agree but can have arbitrary value (and we call such a face {\em free}); or at least one of the the indices along the boundary is one of $j_1,\dots,j_r,i_1,\dots,i_r$, in which case all the indices along the boundary of the face must have certain fixed values (which we call {\em fixed faces}).  Thus
$$ \sum_{J_1,\dots,J_r} \mathbb{E}\big(A_1(j_1,J_1,i_1) A_2(j_2,J_2,i_2) \cdots A_r(j_r, J_r, i_r)\big) = \sum_{\mathscr{T}} C_\mathscr{T} N^{-d/2} N^{F(\mathscr{T})-f(\mathscr{T})}$$
where $C_\mathscr{T}$ is a constant depending only on $\mathscr{T}$, $F(\mathscr{T})$ refers to the total number of faces and $f(\mathscr{T})$ to the number of fixed faces in a triangulation $\mathscr{T}$. Since the number of edges in the triangulation is $d/2$ (because the number of half-edges is $d$), Euler's formula for genus gives us the inequality $$ 1 - d/2 + F(\mathscr{T}) \leq 2$$ so that $$F(\mathscr{T}) -d/2 -1 \leq 0.$$  Since there must be at least one fixed face, $f(\mathscr{T})\geq 1$ and thus $$\lim_{N\to\infty} N^{F(\mathscr{T}) - d/2 - f(\mathscr{T})} $$  is zero unless $f(\mathscr{T})=1$ and the genus is $2$.

If $f(\mathscr{T})=1$, it must be that all half-edges labeled by  $j_1,\dots,j_r,i_1,\dots,i_r$  bound the same face. But this means that $$
\mathbb{E} \prod_{s=1}^r [ q_s(A_1(N),\dots,A_k(N))]_{j_s, i_s} = \prod_{s=1}^r \mathbb{E} [ q_s(A_1(N),\dots,A_k(N))]_{j_s, i_s}.$$  It remains to note that by unitary invariance and \cite{DVV:random} $$\mathbb{E} [ q_s(A_1(N),\dots,A_k(N))]_{j_s, i_s} = \frac{1}{N} {Tr} (q_s(A_1(N),\dots,A_k(N)) \Sigma(j_s,i_s))\to \delta_{j_s,i_s} \tau(q_s)$$  where $\tau$ is the free semicircle law and $\Sigma(a,b)$ is the matrix $\sum_{u} E_{a+u,b+u}$ (addition of indices mod $N$).  
\end{proof}

A similar result also holds for matrices of the form $A(N)^*A(N)$ where $A(N)$ is a random Gaussian matrix of size $N\times M(N)$ and $0<\lim_{N\to\infty} M(N)/N <\infty$ normalized so that each column has total variance $1$.  We leave details to the reader.

\subsection{Asymptotic infinitesimal freeness for random matrices} \label{sec:AsymptFree}  We are now ready to state the main result of our paper.
\begin{theorem} \label{thm:asymtinffree}
Let $A_1(N),\dots,A_k(N)$ be the random matrix ensemble described in either \S\ref{sec:notation} or \S\ref{sec:GaussianNotations}.  Then for any polynomial $p$ in variables $a_1,\dots,a_k$ and $(e_{ij}:1\leq i,j\leq N_0)$  the following limits exist: $$\tau(p) = \lim_{N\to\infty} \frac{1}{N} \mathbb{E} Tr\big(p(A_1(N),\dots,A_k(N),\{E_{ij}\}_{i,j=1}^{N_0})\big)$$  
$$\tau'(p)=\lim_{N\to \infty}  \mathbb{E}Tr\big(p(A_1(N),\dots,A_k(N),\{E_{ij}\}_{i,j=1}^{N_0})\big)-N \tau (q).$$   Furthermore, the variables $(a_1,\dots,a_n) $ and $(e_{ij} : 1\leq i,j \leq N)$ are infinitesimally free with respect to $(\tau, \tau')$.
\end{theorem}

\begin{proof}
It is not hard to see that under the hypothesis of the Theorem, the assumptions (i) and  (ii) of Lemma~\ref{lemma:infFreeConditions} are satisfied.    It thus remains to verify (a) and (b) of Lemma~\ref{lemma:infFreeConditions}.

Assume that  $p$ is a monomial involving at least one term $e_{ij}$. By Lemma~\ref{lemma:limit} or Lemma~\ref{lemma:limitGaussian}, $$\lim_{N\to\infty} \mathbb{E}Tr\big(p(A_1(N),\dots,A_k(N),\{E_{ij}\}_{i,j=1}^{N_0})\big)$$ exists.  Thus $\tau(p)=0$, which gives us condition (a).  

To verify condition (b), we need to check  that if $Q_j \in \operatorname{Alg}(a_1,\dots,a_k)$ so that $\tau(Q_j)=0$ and $E_j $ in the non-unital algebra generated by $e_1,\dots,e_m$, $\tau'(E_1 Q_1 E_2 Q_2 \dots E_r Q_r)=0$.  By linearity, we may assume that each $E_j$ is equal to $E_{s(j),t(j)}$.  But by Lemma~\ref{lemma:limit} or Lemma~\ref{lemma:limitGaussian}, the limit of $$\mathbb{E}Tr(E_1 Q_1 E_2 Q_2 \dots E_r Q_r)$$ is a product of terms involving $\tau(Q_j)$ which are assumed to be zero.
\end{proof}

Because infinitesimal freeness is directly related with freeness we obtain the following:

\begin{theorem} \label{thm:free} Let $A_1(N),\dots,A_k(N)$ be a random matrix ensemble described in \S\ref{sec:notation} or \S\ref{sec:GaussianNotations}. 
Let $\nu_N$ be the joint law of the matrices $( E_{ij}: {1\leq i,j\leq N_0})$ regarded as elements of the non-commutative probability space $(M_{N\times N},\frac{1}{N} Tr)$, and let $\tau$ be the limit law of matrices $A_1(N),\dots,A_k(N)$ as $N\to \infty$.  

Let $(A_1,\dots,A_k) \cup ( e_{ij} :1\leq i,j\leq N_0)\in (B,\phi_N)$ be non-commutative variables so that $(A_1,\dots,A_k)$ have law $\tau$, $(e_{ij}:1\leq i,j\leq N_0)$ have law $\nu_N$, and $(A_1,\dots,A_k)$ are free with respect to $\phi_N$ from $(e_{ij}:1\leq i,j\leq N_0)$ (in other words, $\phi_N = \tau * \nu_N$).   

Then for any non-commutative polynomial $p$ in variables $(a_1,\dots,a_n)\cup (e_{ij}:1\leq i,j\leq N_0)$, $$
\mathbb{E} \frac{1}{N} Tr\left(p (A_1(N),\dots,A_k(N),(E_{ij})_{i,j=1}^{N_0}  )\right) - \phi_N \left(p (A_1,\dots,A_k,(e_{ij})_{i,j=1}^{N_0}  )\right) = o(1/N).
$$ 
\end{theorem}

This theorem shows that the $1/N$ correction to the limit law of $A_1,\dots,A_k$ occasioned by perturbing them in an arbitrary way by  finite-rank matrices $E_{ij}$ can be computed purely in terms of free probability: the correction to order $1/N$ is the same if we were perturbing the matrices $A_1,\dots,A_k$ by {\em freely independent} matrices $e_{ij}$ viewed in as elements of the space of $N\times N$ matrices by padding them with zeros.  This result in a way explains the occurrence of free probability machinery (such as subordination functions) in the analysis of outliers in the spectra of spiked matrices, see e.g. \cite{spikes,spikes4}.

For an $N\times N$ self-adjoint random matrix $A$ let us denote by $\eta^A$ the empirical spectral measure $\mathbb{E} \frac{1}{N} \sum_j \delta_{\lambda_j}$ where $\lambda_1,\dots,\lambda_N$ are the eigenvalues of $A$.  Let  $F$ be a fixed self-adjoint operator of rank $N_0$ and eigenvalues $\lambda_1,\dots,\lambda_{N_0}$.  Let $\nu_0 = N_0^{-1} \sum \delta_{\lambda_j}$.

\begin{corollary}
Let $A(N)$ be a random matrix belonging to one of the ensembles described in  \S\ref{sec:notation} or \S\ref{sec:GaussianNotations} and let $\eta^A=\lim_{N\to\infty} \eta^A_N$. Then with the above notations, 
\begin{eqnarray*}
\eta_N^{A+F}  &=&  \mu \boxplus \left( \frac {N_0}{N} \nu_0 - \frac{N-N_0}{N} \delta_0\right)  + o(1/N).
\end{eqnarray*}
\end{corollary}

\section{Some computations: spiked additive and multiplicative perturbations.}
Let us now fix deterministic self-adjoint matrices $\Lambda(N)$ with eigenvalues $\lambda_1(N)\leq \dots \leq \lambda_N(N)$.  We assume that $\sup_{i,N} |\lambda_i (N)| <\infty$ and that $\mu_N = N^{-1} \sum \delta_{\lambda_j(N)}$ converge weakly to a measure $\mu$.  Depending on our context, we will either put $A_N = U\Lambda(N)U^*$, where $U$ is  Haar-distributed random unitary, or assume that $A_N$ is a real or complex self-adjoint Gaussian matrix.  Let us also put $B = \sum_{j=1}^{N_{0}} \theta_j E_{jj}$.  

\subsection{Additive case.}
Since $A_N$ and $B$ are asymptotically infinitesimally free, we can say that the spectral measure $\eta_N$ of $A_N + B$ will satisfy
$$
\eta_N = \eta + \frac{1}{N} \eta' + o(N^{-1})
$$
where $\eta$ and $\eta'$ can be computed using type B free convolution \cite[Proposition 20]{belinschi-shlyakht:typeB} (the formula for type B convolution can be easily obtained from the usual subordination formulation of free convolution of two one-parameter families of measures, differentiating in that parameter).  More precisely, we have that the type B law of $A(N)$ is given by $(\mu_1,\mu_1') := (\mu,0)$ and the type B law of $B$ is given by $(\mu_2,\mu_2') := (0, \sum_{j=1}^{N_0} \delta_{\theta_j} - N_0 \delta_0)$.  
Note that $\mu_2' = \partial_t h_2(t)$  where $h_2(t) = \sum_j  \chi[0,\theta_j]$ and the derivative is taken in distribution sense.  Thus from \cite{belinschi-shlyakht:typeB} we see that 
$$ (\eta,\eta') = (\mu,0) \boxplus_B (\delta_0,\sum_{j=1}^{N_0} \delta_{\theta_j} - N_0 \delta_0).$$   

We immediately get that $\eta = \mu$ and that the Cauchy transform of $\eta'$, $g_{\eta'}$ is given by 
$$g_{\eta'}(z) = F'_\mu(z) \left( \sum_{j=1}^{N_0}\frac{1}{F_\mu(z)-\theta_j} - {N_0}G_\mu(z) \right),$$
where $G_\mu(z)$ is the Cauchy transform of $\mu$ and $F_\mu(z) = 1 / G_\mu(z)$.  Moreover, by \cite{belinschi-shlyakht:typeB} we know that $\eta'$ belongs to the distribution space $\mathcal{M}_2$ (cf. \cite{belinschi-shlyakht:typeB}), and in particular is the derivative of a function $h$ which is itself a difference of two monotone functions.

As in \cite{belinschi-shlyakht:typeB} we write
$$
g_{\eta'}(z) = \int \frac{1}{z-t} d\eta'(t) = \partial_z \int \log (z-t) d\eta'(t) = \partial_z \int \frac{1}{z-t} h(t) dz.
$$

We now note that $$F'_\mu(z) \left( \sum_{j=1}^{N_0}\frac{1}{F_\mu(z)-\theta_j} - {N_0}G_\mu(z) \right) = 
\partial_z \int \frac{1}{F_\mu(z)-t} h_2(t)dt$$ which as in \cite{belinschi-shlyakht:typeB} gives us   $$
\int \frac{1}{z-t} h(t)dt = \int \frac{1}{F_\mu(z)-t} h_2(t)dt$$ so that  $$\eta' = \partial_t h.$$
Using the expression we have for $h_2(t)$, we get  $$\int \frac{1}{z-t} h(t) dt = \int \frac{1}{F_\mu(z)-t} \chi_{[0,\theta_j]}
=\sum_j \log \left(1 - \theta_j G_\mu(z)\right).$$

Let us now consider several cases.

\subsubsection{Case when $\mu$ is the semicircle law.}  In this case, $G_\mu(z) = z - \sqrt{z^2-2}$.  The function $G_\mu(z)$ can be analytically extended to the subset of the real line consisting of the interval $(-\sqrt{2},\sqrt{2})$ as well as the complement of its closure:
$$
G_\mu(t) = \begin{cases} 
t - \sqrt{t^2-2}, & |t|>\sqrt{2} \\
t -i \sqrt{2-t^2}, & |t| < \sqrt{2}.
\end{cases}
$$
Note that the equation $G_\mu(\theta'_j) = 1/\theta_j$ only has (a unique) real solution $\theta'_j$ when $|\theta_j| \geq 1/\sqrt{2}$.  Moreover, $|\theta'_j|\geq \sqrt{2}$.  We will assume that $|\theta_j|\neq 1/\sqrt{2}$ for any $j$ for simplicity.

We now recover $h$ using a kind of Stieltjes inversion formula (which, as was mentioned in \cite{belinschi-shlyakht:typeB}, applies in this case) by considering the limit
$$L(t) = -\lim_{s\downarrow 0} \frac{1}{\pi} \sum_j \operatorname{Im} \log (1-\theta_j G_\mu(is+t)) = - \lim_{s\downarrow 0} \frac{1}{\pi} \sum_j \operatorname{Arg} (1-\theta_j G_\mu(is+t)). $$

Thus if $t\in ( -\sqrt{2},\sqrt{2})$, then $$L(t) = -\frac{1}{\pi} \sum_j \operatorname{Arg}\left( 1 - \theta_j (t - i\sqrt{2-t^2})\right).
$$ 

If $|t|>\sqrt{2}$, then unless $t=\theta'_j$ for some $j$,  $\operatorname{Arg} (1-\theta_j G_\mu(is+t))$ converges either to $0$ or $\pi$, depending on the sign of the limit $1-\theta_j G_\mu(is+t)$. 

It follows that $h(t) = \sum_j h^{(j)} (t)$ where $h^{(j)}(t) = a_j(t) + b_j(t)$ where $a_j$ and $b_j$ are determined as follows. 
\begin{itemize}
\item[(a)] $|\theta_j|<1/\sqrt{2}$: In this case, $$a_j(t) =  -\frac{1}{\pi} \sum_j \operatorname{Arg}\left( 1 - \theta_j (t - i\sqrt{2-t^2})\right)  \chi_{[-\sqrt{2},\sqrt{2}]},\qquad b_j=0.$$ In this case $a_j(t)$ is monotone decreasing for $-\sqrt{2} < t< 1/\theta_j$ and is increasing for $1/\theta_j < t < \sqrt{2} $. 
 
 \item[(b)] $|\theta_j| > 1/\sqrt{2}$:  In this case,
 $$a_j(t) =  -\frac{1}{\pi} \sum_j \operatorname{Arg}\left( 1 - \theta_j (t - i\sqrt{2-t^2})\right)  \chi_{[-\sqrt{2},\sqrt{2}]},\qquad b_j(t) = \begin{cases} -1, & t<\theta'_j \\ 0, & t> \theta'_j. \end{cases}$$   In this case $a_j(t)$ is monotone decreasing on $(-\sqrt{2},\sqrt{2})$ and $b_j(t)$ is a monotone increasing function.
 \end{itemize}
 Furthermore, the function $\operatorname{Arg}\left( 1 - \theta_j (t - i \sqrt{2-t^2}) \right)$ has a branch cut at $t=\theta_j'$ if $|\theta_j|>1/\sqrt{2}$. 
 
 From this we can deduce the ``$1/N$ correction'' to the limit law of $A_N+B$.  It is given by the addition of \begin{equation}
 \label{eq:correction} \frac{1}{N} \left( \sum_{j:|\theta_j|>1/\sqrt{2}} \delta_{\theta'_j}  - \sum_{j} \hat{\nu}_j\right),\end{equation}
where 
\begin{equation}\label{eq:nuHat}\hat{\nu}_j = \frac{ \theta_j (t-2\theta_j) 
}
{
(2\theta_j(t-\theta_j) -1) \sqrt{2-t^2}
} \chi_{[-\sqrt{2},\sqrt{2}]}dt\end{equation}
is a probability measure if $\theta_j > 1/\sqrt{2}$ and is a  signed measure of total mass zero otherwise (in the latter case, $\nu_j$ is the difference of two positive measures supported on $[-\sqrt{2},2\theta_j]$ and $[2\theta_j,\sqrt{2}]$, respectively).

\subsubsection{GUE matrices.} Exactly the same computation works if we replace the matrix $A_N$ by a random Gaussian Hermitian matrix; indeed, because the entries of $A_N$ are complex, we have $$\frac{1}{N} \mathbb{E} Tr(p(A_N)) = \tau(p) + O(1/N^2)$$ where $\tau$ is the semicircle law (cf. \cite{Johanssen,Edelmann}).

Figure~\ref{figure:Theta4} presents the results of numerical simulations. 
 \begin{figure} 
 \begin{tabular}{cc}
 \includegraphics[width=0.4\textwidth]{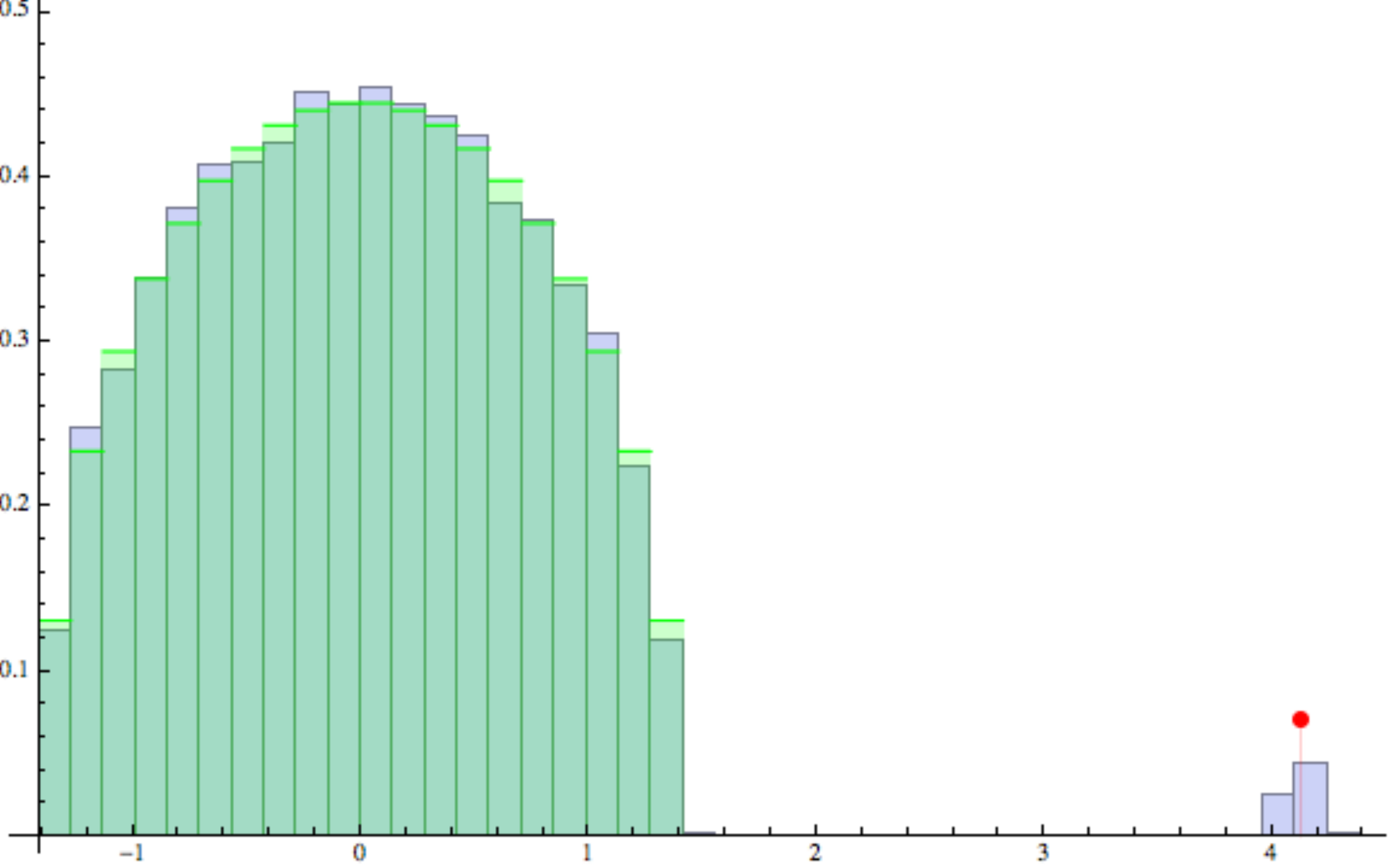} & 
  \includegraphics[width=0.4\textwidth]{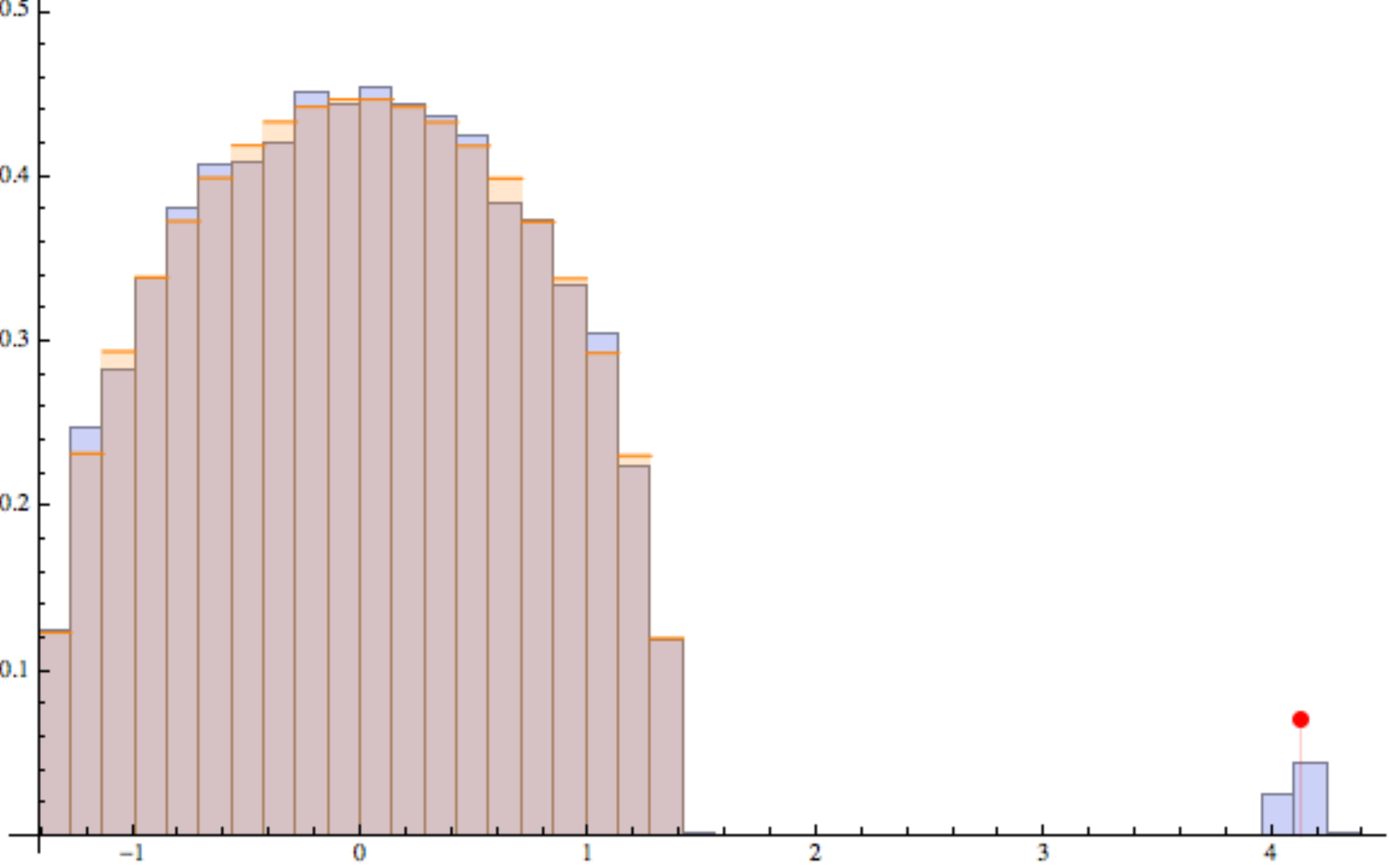}\\
  (1a) & (1b) \\
  \includegraphics[width=0.4\textwidth]{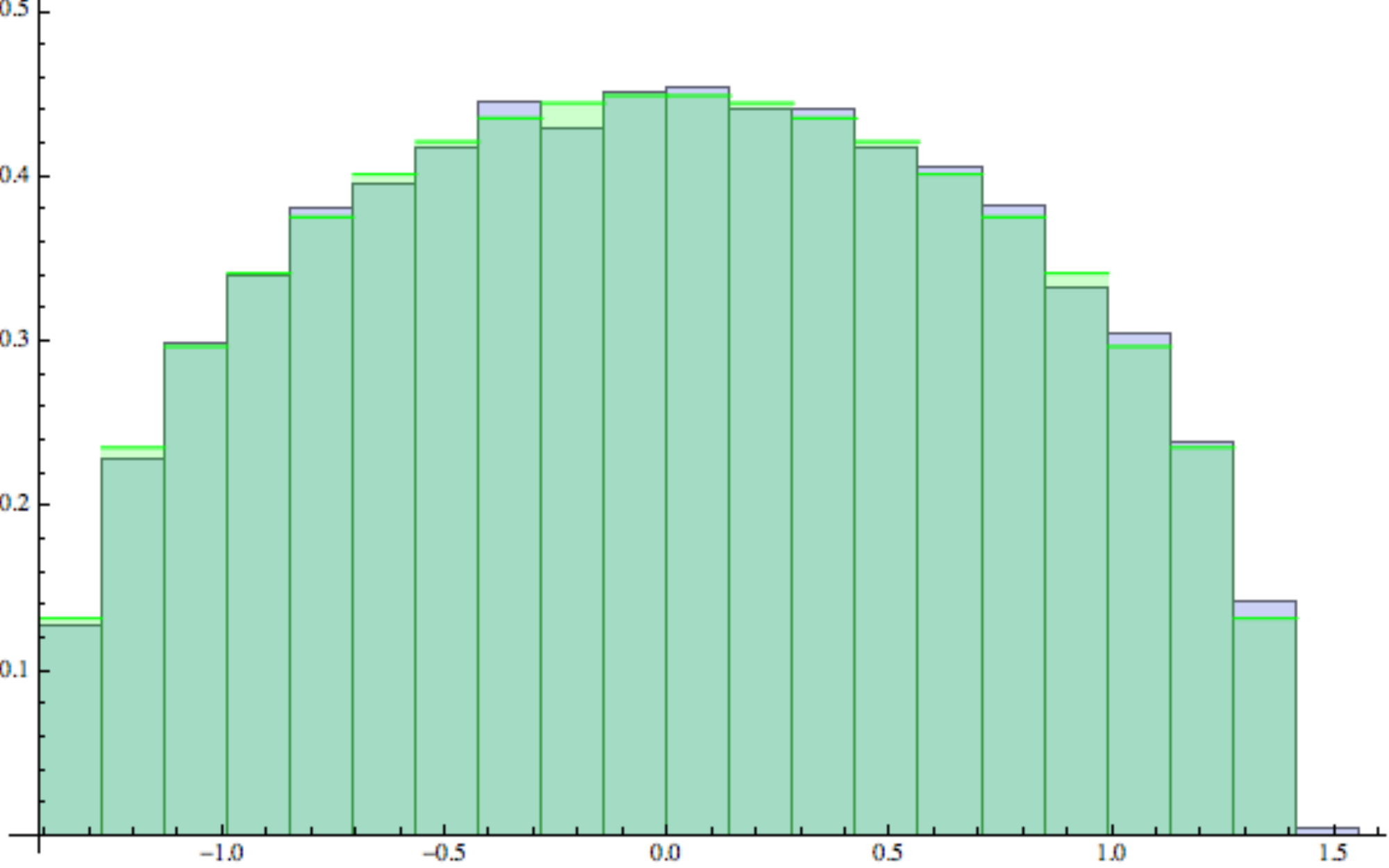} &  
  \includegraphics[width=0.4\textwidth]{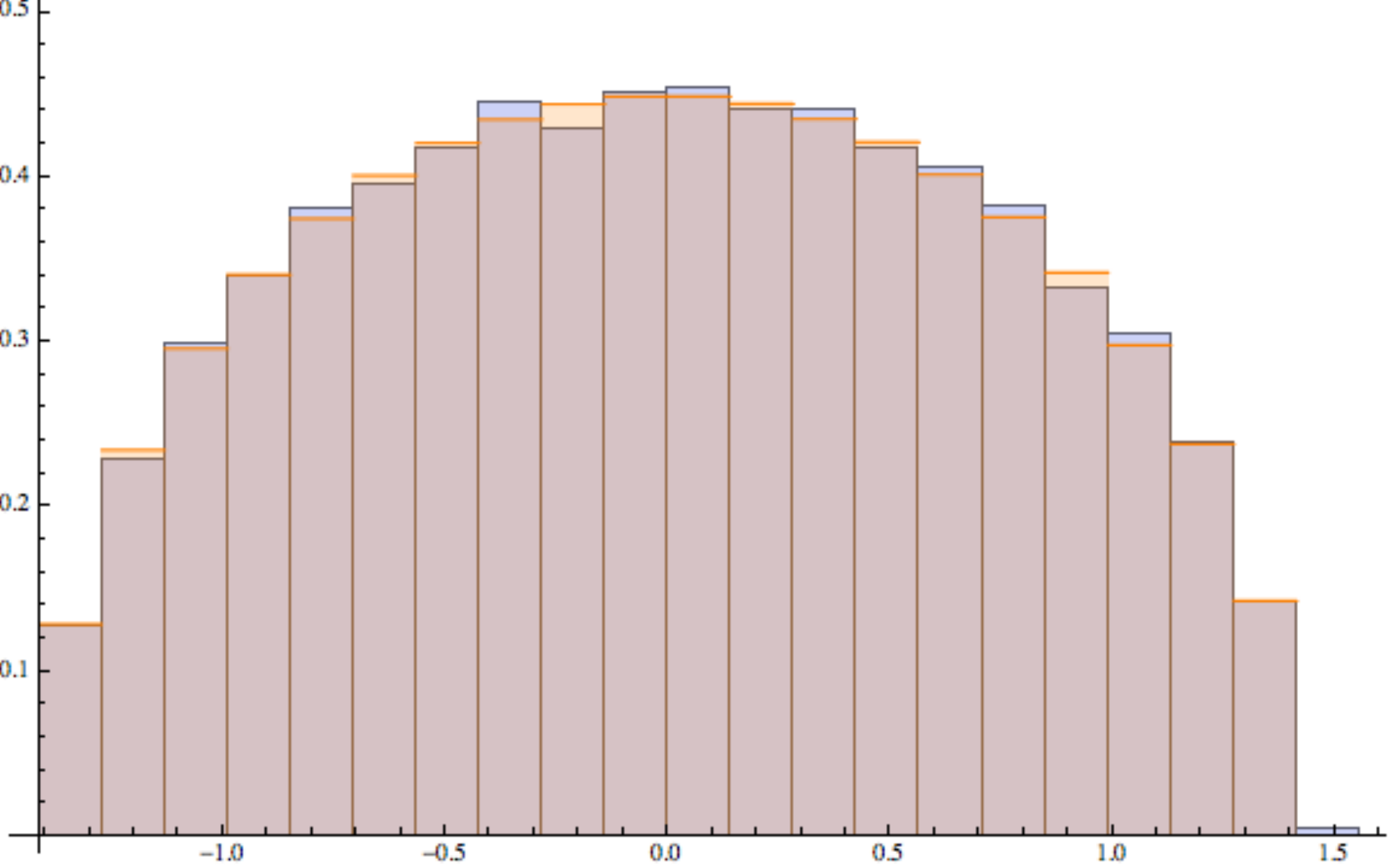} \\
  (1c) & (1d) 
\end{tabular}
 \caption{\label{figure:Theta4} In this simulation $40$ complex Gaussian $100\times 100$ random matrices were chosen at random; each was perturbed by a deterministic rank $1$ matrix with eigenvalue $4$ (in figures 1a, 1b) and $0.4$ (in figures 2a, 2b).  Eigenvalue distributions of the resulting matrices were averaged.  Green represents  predicted eigenvalue density assuming the semicircle law (rescaled by $0.99$ to account for the loss of one eigenvector in figure 1a).  Orange represents predicted eigenvalue density using \eqref{eq:correction}.  Note that our correction \eqref{eq:nuHat} has most of its mass near the edges of the continuous spectrum, giving an improved prediction for the number of eigenvalues near those edges. }
\end{figure}

\subsubsection{GOE matrices}  Let us now consider the case that $A_N$ is a randomly chosen symmetric real matrix.  This time the type B law of $A_N$ is given by  $(\mu_1,\mu_1') = (\mu, \sigma)$ where $\mu$ is the semicircle law and $$\sigma = \frac{1}{4}(\delta_{\sqrt{2}} + \delta_{-\sqrt{2}}) - \frac{1}{2\pi \sqrt{2-t^2}}\chi_{[-\sqrt{2},\sqrt{2}]}(t)dt$$ (cf.  \cite{Johanssen,Edelmann}, case $\beta=1$).  
The Cauchy transforms are:
\begin{eqnarray*}
G_\mu(z) &=& z - \sqrt{2-z^2}\\
g_\sigma(z) &=& \frac{1}{4 (z-\sqrt{2})} + \frac{1}{4(z+\sqrt{2})} -\frac{1}{2\sqrt{z^2-2}}.
\end{eqnarray*}
We therefore obtain that the law $\eta_N$ of $A_N+B$ will satisfy
$$\eta_N = \eta + \frac{1}{N} \eta' + o(N^{-1})$$ where 
$$(\eta, \eta') = (\mu,\sigma) \boxplus_B (\delta_0, \sum_{j=1}^{N_0} \delta_{\theta_j} - N\delta_0).
$$
From this we get immediately that $\eta$ is again the semicircle law, and $$
g_{\eta'}(z) = F'_\mu(z) \left(\sum_{j=1}^{N_0} \frac{1}{F_\mu(z) - \theta_j} - N_0 G\mu(z)\right) + 
g_\sigma(z).$$  
Repeating much the same analysis as before, we get that the ``$1/N$ correction'' to the limit law of $A_N+B$:  it is given by the addition of $$\frac{1}{N} \left[\left( \sum_{j:|\theta_j|>1/\sqrt{2}} \delta_{\theta'_j}  - \sum_{j} \hat{\nu}_j\right) + \frac{1}{4}\left(\delta_{-\sqrt{2}}+\delta_{\sqrt{2}}\right) -\frac{1}{2\pi \sqrt{2-t^2}}\chi_{[-\sqrt{2},\sqrt{2}]}(t) dt\right],$$ where as before  $\theta'_j$ is the solution to  $G_\mu(\theta'_j) = 1/\theta_j$ if $|\theta_j| > 1/\sqrt{2}$ and $\hat{\nu}_j$ is given by \eqref{eq:nuHat}.

\subsubsection{Case when $\mu$ is a discrete measure.}

Let us assume now that $\mu = \frac{1}{K} \sum_{i=1}^K \delta_{\lambda_i}.$  In this case we can once again extend $G_\mu$ to the complement of the set $\{\lambda_j : 1\leq j\leq K\}$ on the real axis by the formula $$G_\mu(t) = \frac{1}{K} \sum_{i=1}^K \frac{1}{t-\lambda_j}.$$  The equation $G_\mu(\theta') = 1/\theta$ may now have multiple solutions.  
We are once again led to consider $$
L(t)=\lim_{s\downarrow 0} \frac{1}{\pi} \sum_j \operatorname{Arg} (1-\theta_j G_\mu(is+t)).
$$
Since $G_\mu(z)$ becomes real when $\operatorname{Im}(z)\to 0$, it follows that $L(t)$ is either $0$ or $-\pi$, depending on the sign of $1-\theta_j G_\mu(t)$.  As $t$ runs from $-\infty$ to $\infty$, this sign changes each time we either go through a pole of $G_\mu(z)$ (i.e., $t=\lambda_i$ for some $i$) or if $t=\theta'$, where $\theta'$ solves $G_\mu(\theta') = 1/\theta$. 

It follows that the law of $A_N+B$ is approximately given by

$$\frac{N-\nu(1)}{N} \mu + \frac{1}{N} \nu$$
where $\nu$ is the measure putting mass $1$ on each  solution $\theta'$ to $G_\mu(\theta') = 1/\theta$.

\subsection{Multiplicative case.}

Let us now assume that $C  = \Sigma B^* B \Sigma$ where $B$ is an $N\times p$ non-selfadjoint complex Gaussian random matrix with iid entries of variance $2/\sqrt{N}$ and $\Sigma$ is a deterministic matrix of size $N\times N$ having $N_0$ eigenvalues $0<\theta_1^{1/2}\leq \theta_2^{1/2} \dots \leq \theta_{N_0}^{1/2}$ and having the rest of the eigenvalues equal to $1$.  We assume that $p/N\to \lambda \in (0,+\infty)$ as $N\to\infty$.  One can again check, as we did with self-adjoint Gaussian random matrices that $B^*B$ and $\Sigma$ are asymptotically infinitesimally free.  We can again use type B convolution to compute the law of $C$ to order higher than $1/N$.  

Let $\eta^C_N$ be the empirical spectral measure of $C$ and write $\eta^C_N = \eta + \frac{1}{N} \eta' + o(1/N)$.  Let $\mu$ be the limiting eigenvalue distribution of $B^*B$; it is the free Poisson (also called Marchenko-Pastur) law of parameter $\lambda$.  The type B laws of $B^*B$ and $\Sigma^2$ are given by, respectively, by:
$$ (\mu_1,\nu_1) = (\mu, 0) \qquad\textrm{and}\qquad (\mu_2,\nu_2) = (\delta_1, \sum_{j=1}^{N_0} (\delta_{\theta_j} - \delta_1)).$$

Because of infinitesimal freeness, we get that $$(\eta, \eta') = (\mu_1,\nu_1)\boxplus_B (\mu_2,\nu_2).$$

Let us denote by $\psi$ the $\psi$-transform, related to the Cauchy transform: $$\psi_\nu = \int \frac{tz}{1-tz} d\nu(t) = \frac{1}{z} G_\nu \left(\frac{1}{z}\right) -1.$$

Then $$\psi_{\mu_2} = \frac{z}{1-z},\qquad \psi_{\nu_1} = 0.$$

The appropriate analogs of subordination functions for multiplicative convolution (see e.g. the last section of \cite{belinschi-shlyakht:typeB}) satisfy: $$\omega_1(z) = z, \qquad \frac{\omega_2(z)}{1-\omega_2(z)} = \psi_\mu(z).$$

By  \cite{belinschi-shlyakht:typeB}  we immediately get: $$\eta = \mu_1\boxplus \mu_2 = \mu$$ and 
\begin{eqnarray*}
\frac{\psi_{\eta'}(z)}{z} &=& \frac{\psi_{\nu_2}(\omega_2(z))}{\omega_2(z)} \omega'(z) =\int \frac{t}{1-\omega_2(z) t}\omega_2'(z) d\nu_2(t) 
=  -\partial_z \int \log (1 - \omega_2(z)t) d\nu_2(t) \\
&=& -\partial_z \sum_j \log \left( \frac{ 1 - \omega_2(z) \theta_j}{1-\omega_2(z)} \right) = -\partial_z \sum_j \log \left( 1 + \psi_\mu (z) (1-\theta_j )\right). 
\end{eqnarray*}
The left-hand size is the same as $$\int \frac{t}{1-zt} d\eta'(t),$$ so that if we suppose that (in the sense of distributions) $d\eta'(t) = h'(t) dt$, we get that the left hand side is $$\frac{\psi_{\eta'}(z)}{z}  = \int \frac{t}{1-zt} h'(t) dt =  \partial_z \int \frac{z}{1-zt} h(t) dt.$$ From this we get:$$
\int \frac{1}{1/z-t} h(t) dt =- \sum_j \log \left( 1 + \psi_\mu (z) (1- \theta_j )\right).
$$
Substituting $1/z$ for $z$ we get
\begin{eqnarray*}
G_{\eta'}(z) &=& \int \frac{1}{t-z} h(t)dt =  \sum_j \log \left(1+\psi_\mu(1/z) (1-\theta_j)\right) 
\\ &=&  \sum_j \log \left(1 + \psi_\mu(1/z) (1-\theta_j)\right).
\end{eqnarray*}
Once again, we can recover $h$ using a kind of Stietjes inversion formula, involving 
\begin{eqnarray*}
L(t) &=& -\lim_{s\downarrow 0} \frac{1}{\pi} \sum_j \operatorname{Im} \log \left(1 + \psi_\mu(1/z) (1-\theta_j)\right) \\ &=& - \lim_{s\downarrow 0} \frac{1}{\pi} \sum_j \operatorname{Arg} \left(1 + \psi_\mu(1/(s+it))(1-\theta_j)\right). 
\end{eqnarray*}
For $\mu$ the free Poisson law of parameter $\lambda$, $\phi_\mu(1/z)=z G_\mu(z)-1$ can be continued analytically to the real axis excluding the set $\{ (1-\sqrt{\lambda})^2, (1+\sqrt{\lambda})^2\}$:
$$
\psi_\mu(1/t)  = \frac{1}{2}\begin{cases} 
-1-\lambda +t - \sqrt{(1+\lambda -t)^2 -4\lambda}, & t\notin [(1-\sqrt{\lambda})^2, (1+\sqrt{\lambda})^2];
 \\[4pt]
-1-\lambda +t - i\sqrt{4\lambda - (1+\lambda -t)^2 }, & \textrm{otherwise}.
\end{cases}
$$
The analysis is now similar to the additive case, so we only provide a brief sketch.

If $t$ is outside the interval  $[(1-\sqrt{\lambda})^2, (1+\sqrt{\lambda})^2]$, then $L(t)$ is either $0$ or $\pi$ depending on the sign of $1 + \psi_\mu(1/t) (\theta_j-1)$, which changes when $t=\theta_j'$ is a solution to $ 1 + \psi_\mu(1/z) (1-\theta_j) = 0$.  This equation is equivalent to $$\frac{1 -\omega_2(1/\theta'_j) \theta_j } { 1-\omega_2(1/z)}  =0$$ which in turn is equivalent to $$\omega_2(1/\theta'_j) = 1/\theta_j$$ (compare Theorem 2.3 of \cite{spikes}).  

It follows that the $1/N$ correction $\eta'$ to $\eta$ consists of a certain (possibly signed) measure supported on $[(1-\sqrt{\lambda})^2, (1+\sqrt{\lambda})^2]$ (whose density is the derivative of $L(t)$), together with point masses at each solution to $$\omega_2(1/\theta'_j) = 1/\theta_j$$
in accordance with the results of \cite{spikes1,spikes4,spikes}.

\end{document}